\documentclass[conference]{IEEEtran}
\IEEEoverridecommandlockouts
\usepackage{amsmath,amssymb,amsfonts}
\usepackage{amsthm}
\usepackage{graphicx}
\usepackage{textcomp}
\usepackage{mathptmx}
\usepackage{amsmath}
\usepackage{tabularx}
\usepackage{subfigure}
\usepackage{xcolor}
\usepackage{cite}
\usepackage{multirow}
    
\graphicspath{{./images/}}
    
\newtheorem{assumption}{Assumption}

\newtheorem{proposition}{Proposition}
\newtheorem{lemma}{Lemma}

\usepackage{url} 
\usepackage{algorithm,algpseudocode}

\def\BibTeX{{\rm B\kern-.05em{\sc i\kern-.025em b}\kern-.08em
    T\kern-.1667em\lower.7ex\hbox{E}\kern-.125emX}}

\begin{document}

\title{\huge{Fair and Distributed Dynamic Optimal Transport for\\ Resource Allocation over Networks}
\vspace{-2mm}}

\author{Jason Hughes and Juntao Chen
\thanks{The authors are with the Department of Computer and Information Science, Fordham University, New York, NY, 10023 USA. E-mail: \{jhughes50,jchen504\}@fordham.edu}}

\maketitle

\begin{abstract}
Optimal transport is a framework that facilitates the most efficient allocation of a limited amount of resources. However, the most efficient allocation scheme does not necessarily preserve the most fairness. In this paper, we establish a framework which explicitly considers the fairness of dynamic resource allocation over a network with heterogeneous participants. As computing the transport strategy in a centralized fashion requires significant computational resources, it is imperative to develop computationally light algorithm that can be applied to large scale problems. To this end, we develop a fully distributed algorithm for fair and dynamic optimal transport with provable convergence using alternating method of multipliers. In the designed algorithm, each corresponding pair of resource supplier and receiver compute their own solutions and update the transport schemes through negotiation iteratively which do not require a central planner. The distributed algorithm can yield a fair and efficient resource allocation mechanism over a network. We corroborate the obtained results through case studies.
\end{abstract}


\section{Introduction}

Optimal transport (OT) is a centralized framework that enables the design of efficient schemes for distributing resources by considering heterogeneous constraints between the resource suppliers and receivers \cite{galichon2018optimal}. Efficiency in transporting and distributing resources has long been sought out, such as the optimal dispatch of raw materials in manufacturing, backup of power units in disaster affected neighbors, and matching between employees and tasks in an enterprise network. 

Under the standard OT paradigm, the resource distribution scheme maximizes the aggregated utilities of all participants in a centralized way, regardless of whether that distribution is fair for the resource receivers \cite{you2016energy,zhang2019consensus}. This efficiency maximization paradigm is not suitable for many societal problems. For example, in energy systems, the resilience planning should take into account these generally under considered communities which are hit heavily by natural disasters. Though from the central planner's perspective, the resilience planning in these areas may not contribute as significant as other areas to the system's utility by cost-benefit analysis. Therefore, it is necessary to incorporate fairness during the transport mechanism design for constrained resource allocation, especially in the scenarios that promote social equity.



The transport network that the resources are distributed over becomes more complex with a large number of suppliers and receivers. This large-scale feature of the OT problem gives rise to another concern on the centralized computation of the transport plan. The required computation for centralized planning grows exponentially with the number of participants in the framework. This concern is further intensified if the resource allocation is completed over a period of time sequentially. To this end, we aim to develop a distributed algorithm for fair and efficient dynamic resource allocation where the centralized planner is not necessary. The distributed algorithm is obtained by leveraging alternating method of multipliers (ADMM) approach \cite{boyd2011distributed}.

To enable a fair resource allocation, we include a fairness measure in the objective function in the dynamic OT framework. Therefore, the resulting dynamic transport plan will have a balance between efficiency and fairness. In the designed ADMM-based distributed algorithm, each participant (resource supplier or receiver) only needs to solve its own problem and exchange the results with the corresponding connected agents, which enables parallel updates on the solution. The algorithm terminates when the solution computed at each pair of supplier and receiver coincides, at which point the dynamic transport strategy given by our developed distributed algorithm is the same as the one under centralized design. 

Our distributed algorithm offers insights for fair and efficient dynamic resource distribution over networks. First, the updates of transport strategies at both the supplier side and the receiver side can be seen as bargaining for the resources transfer. The bargaining process ends when both parties reach an agreement. Furthermore, during each update, each receiver node in the network proposes a solution that explicitly considers the fairness. In comparison, the supplier nodes solely focus on maximizing their payoff by selling their resources. At the next round of updates, each pair of supplier and receiver will propose a resource distribution scheme that is closer to the average of their previous solutions. It indicates that, as the bargaining progresses, the resource suppliers will also consider the fairness and the receivers will take into account the efficiency of the transport plan to have a consensus. The algorithm can also be implemented online conveniently by adapting to the changes in the resource allocation network and participants’ preferences.

The contributions of this paper are summarized as follows. First, we establish a framework that can yield fair and efficient dynamic resource transportation over networks. Second, we develop a distributed algorithm based on ADMM to compute the dynamic transport strategy in which the resource suppliers and receivers negotiate iteratively on the strategy. Third, we use case studies to corroborate the effectiveness and applicability to changing environment of the algorithm.







\textit{Related Works:} 
Optimal resource allocation/matching has been investigated vastly in various fields, including communication networks \cite{you2016energy}, energy systems \cite{awad2016optimal}, critical infrastructure \cite{huang2018distributed} and cyber systems \cite{chen2018security}. To compute the optimal transport strategy efficiently, a number of techniques have been developed, such as simultaneous approximation \cite{mirrokni2012simultaneous}, population-based optimization \cite{deb2017population}, and distributed algorithms \cite{zhang2019consensus,niu2013efficient}.
Our work is also related to fair allocation of constrained resources \cite{coucheney2009fair,abdel2014utility}. In this work, we leverage ADMM to develop a fast computational mechanism for fair and efficient dynamic resource matching over large-scale networks.


\section{Problem Formulation}\label{sec:problem}
In this section, we first present a standard dynamic OT framework for limited resource allocation over a network. Then, we extend the framework to a fair OT setting by considering the fairness in the dynamic transport design.

In the network, we denote $\mathcal{X}:=\{1, ..., |\mathcal{X}|\}$ the set of destinations or targets that receive the resources, and $\mathcal{Y}:=\{1, ..., |\mathcal{Y}|\}$ the set of origins or sources that distribute resources to the targets. Specifically, each source node $y\in\mathcal{Y}$ is connected to a number of target nodes denoted by $\mathcal{X}_y$, representing that $y$ has choices in allocating its resources to a specific group of destinations $\mathcal{X}_y$ in the network. Similarly, it is possible that each target node $x\in\mathcal{X}$ receives resources from multiple source nodes, and this set of suppliers to node $x$ is denoted by $\mathcal{Y}_x$. Note that $\mathcal{X}_y$, $\forall y$ and $\mathcal{Y}_x$, $\forall x$ are nonempty. Otherwise, the corresponding nodes are isolated in the network and do not play a role in the considered optimal transport strategy design. It is also straightforward to see that the resources are transported over a bipartite network, where one side of the network consists of all source nodes and the other includes all destination nodes. This bipartite graph may not be complete due to constrained matching policies between participants. Another reason yielding incomplete bipartite graph in practice can be the infeasible transport of resources between certain pairs of source and destination nodes incurred by long transport distance. For convenience, we denote by $\mathcal{E}$ the set including all feasible transport paths in the network, i.e., $\mathcal{E}:=\{\{x,y\}|x\in\mathcal{X}_y,y\in\mathcal{Y}\}$. Note that $\mathcal{E}$ also refers to the set of all edges in the established bipartite graph for resource transportation.

We denote by $\pi_{xy}^t\in\mathbb{R}_+$ the amount of resources transported from the origin node $y\in\mathcal{Y}$ to the destination node $x\in\mathcal{X}$ at time $t$, where $\mathbb{R}_+$ is the set of nonnegative real numbers and $t\in\mathcal{T}:=\{1,2,...,T\}$. For convenience, let $\Pi:=\{\pi_{xy}^t\}_{x\in\mathcal{X}_y,y\in\mathcal{Y},t\in\mathcal{T}}$ be the transport plan designed for the considered network. To this end, the centralized dynamic optimal transport problem can be formulated as follows:
\begin{align}
    \max_{\Pi}\ \sum_{t\in\mathcal{T}}\sum_{x\in\mathcal{X}} \sum_{y\in\mathcal{Y}_x}& d_{xy}(\pi_{xy}^t) + \sum_{t\in\mathcal{T}} \sum_{y\in\mathcal{Y}} \sum_{x\in\mathcal{X}_y} \left(s_{xy}(\pi_{xy}^t) - c_{xy}(\pi_{xy}^t)\right)\notag\\
    \mathrm{s.t.}\quad &\underline{p}_{x}\leq \sum_{t\in\mathcal{T}} \sum_{y\in\mathcal{Y}_x} \pi_{xy}^t\leq \bar{p}_{x},\ \forall x\in\mathcal{X},\notag\\
    &\underline{q}_{y}\leq \sum_{t\in\mathcal{T}}\sum_{x\in\mathcal{X}_y} \pi_{xy}^t\leq \bar{q}_{y},\ \forall y\in\mathcal{Y},\label{OT1:eqn}\\
    &\pi_{xy}^t\geq 0,\ \forall t\in\mathcal{T},\ \forall\{x,y\} \in\mathcal{E},\notag
\end{align}
where $d_{xy}:\mathbb{R}_+\rightarrow\mathbb{R}$ and $s_{xy}:\mathbb{R}_+\rightarrow\mathbb{R}$ are utility functions for destination/target node $x$ and source node $y$, respectively; $c_{xy}:\mathbb{R}_+\rightarrow\mathbb{R}$ is a cost function of source node $y$ for transporting resources to target node $x$. Furthermore, $\bar{p}_x\geq \underline{p}_{x}\geq 0$, $\forall x\in\mathcal{X}$ and $\bar{q}_y\geq \underline{q}_{y}\geq 0$, $\forall y\in\mathcal{Y}$. The constraints $\underline{p}_{x}\leq \sum_{t\in\mathcal{T}}\sum_{y\in\mathcal{Y}_x} \pi_{xy}\leq \bar{p}_{x}$ and $\underline{q}_{y}\leq \sum_{t\in\mathcal{T}}\sum_{x\in\mathcal{X}_y} \pi_{xy}\leq \bar{q}_{y}$ capture the limitations on the amount of requested and transferred resources at the target $x$ and source $y$, respectively. 

We have the following assumption on the utilities functions $d_{xy}$ and $s_{xy}$ and the cost function $c_{xy}$.
\begin{assumption}\label{assump:1}
The utility functions $d_{xy}$ and $s_{xy}$ are concave and monotonically increasing, and the transport cost function $c_{xy}$ is convex and monotonically increasing on $\pi_{xy}^t$, $\forall x\in\mathcal{X},\forall y\in\mathcal{Y}$.
\end{assumption}

There are a number of functions of interest that satisfy the properties in Assumption \ref{assump:1}. For example, the utility functions $d_{xy}$ and $s_{xy}$ can adopt a linear form, indicating a linear growth of payoff on the amount of transferred and consumed resources. $d_{xy}$ and $s_{xy}$ can also take a logarithmic form, representing the marginal utility decreases with the amount of transported resources. The cost function $c_{xy}$ can admit linear and quadratic forms, capturing the flat and increasing growth of transport costs on the resources, respectively.

In the above formulation, there is no consideration of fairness in resource allocation. The central planner devises an optimal transport strategy by maximizing the social welfare. In practice, some target nodes may not contribute as significant as other nodes to the social objective by receiving a certain amount of resources from the sources. This efficient resource allocation plan yields a larger objective value. However, it is not fair for some nodes if their requests for resources are ignored. Therefore, it is urgent to incorporate the equity consideration during resource allocation. One possible way to achieve this goal is to introduce a fairness measure to the objective function in the OT framework as follows:
\begin{equation}\label{obj_fair:eqn}
\begin{aligned}
    \sum_{t\in\mathcal{T}}\sum_{x\in\mathcal{X}} \sum_{y\in\mathcal{Y}_x} d_{xy}(\pi_{xy}^t) +\sum_{t\in\mathcal{T}} \sum_{y\in\mathcal{Y}} \sum_{x\in\mathcal{X}_y} \left(s_{xy}(\pi_{xy}^t) 
    - c_{xy}(\pi_{xy}^t)\right)\\
    +\sum_{x\in\mathcal{X}} \omega_x f_{x}(\sum_{t\in\mathcal{T}}\sum_{y\in\mathcal{Y}_x}\pi_{xy}^t),
\end{aligned}
\end{equation}
where $\omega_x\geq 0$ is a weighting constant for fairness, and $f_x:\mathbb{R}_+\rightarrow \mathbb{R}$. Note that $\sum_{t\in\mathcal{T}}\sum_{y\in\mathcal{Y}_x}\pi_{xy}^t$ is total amount of resources received for the target node $x$ over $T$ periods of time. Thus, $f_{x}(\sum_{t\in\mathcal{T}}\sum_{y\in\mathcal{Y}_x}\pi_{xy}^t)$ quantifies the level of fairness by allocating $\sum_{t\in\mathcal{T}}\sum_{y\in\mathcal{Y}_x}\pi_{xy}^t$ resources to each target $x$. To facilitate a fair transport strategy, the central planner needs to devise $f_x$ strategically. One consideration is that the marginal utility of the fairness term $f_x$ should decrease. Otherwise, it will lead to an unfair distribution of resources, i.e., some target nodes receive most of the resources in the network as the central planner aims to maximize $\sum_{x\in\mathcal{X}} \omega_x f_{x}(\sum_{t\in\mathcal{T}}\sum_{y\in\mathcal{Y}_x}\pi_{xy}^t)$.

We have the following assumption on the property of fairness function.
\begin{assumption}\label{assump:2}
The fairness function $f_x$, $\forall x\in\mathcal{X}$ is concave and monotonically increasing. 
\end{assumption}

There can be various choices for the fairness function. One possible choice is a proportional fairness function \cite{abdel2014utility}:
\begin{equation}\label{fairness:eqn}
    f_x\Big(\sum_{t\in\mathcal{T}}\sum_{y\in\mathcal{Y}_x}\pi_{xy}^t\Big) = \log\Big(\sum_{t\in\mathcal{T}}\sum_{y\in\mathcal{Y}_x}\pi_{xy}^t+1\Big), \ \forall x\in\mathcal{X}.
\end{equation}
To this end, the central planner's goal is to devise a fair and efficient transport strategy that maximizes the objective function \eqref{obj_fair:eqn} while takes into account the same set of constraints on resources capacity in \eqref{OT1:eqn}.


\section{Distributed Algorithm for Fair and Efficient Dynamic Transport Strategy Design}\label{sec:algorithm}
The planner can solve the formulated optimization problem in Section \ref{sec:problem} in a centralized manner. One primal concern is the computational feasibility. It can be computationally expensive to calculate the fair and efficient resource distribution plan when the number of sources and targets becomes enormous as can be observed in a large-scale network. Therefore, we shift our attention in finding a fair and efficient transport strategy from a centralized way to a fully distributed fashion.

\subsection{Feasibility and Optimality}
Before developing the distributed algorithm, we first analyze the feasibility of the formulated optimization problem. 
\begin{lemma}\label{lem:fea}
It is feasible to find a fair transport plan $\Pi$ if the following conditions are satisfied:
\begin{align}
    \sum_{y\in\mathcal{Y}_x}\bar{q}_y &\geq \underline{p}_x,\quad \forall x\in\mathcal{X},\label{fea:eqn1}\\
    \sum_{y\in\mathcal{Y}} \bar{q}_y & \geq \sum_{x\in\mathcal{X}} \underline{p}_x.\label{fea:eqn2}
\end{align}
\end{lemma}
The two inequalities in Lemma \ref{lem:fea} have natural interpretations. \eqref{fea:eqn1} ensures that all the target nodes' requests can be fulfilled. \eqref{fea:eqn2} indicates the the total demand of resources is less than the total supply that the source nodes can provide.

We next characterize the existence of optimal solution to the formulated problem.

\begin{lemma}
Under Assumptions \ref{assump:1} and \ref{assump:2}, and the inequalities \eqref{fea:eqn1} and \eqref{fea:eqn2}, there exists a fair and efficient transport strategy that maximizes the objective \eqref{obj_fair:eqn} while satisfying the constraints in \eqref{OT1:eqn}.
\end{lemma}
The existence of the optimal solution is guaranteed by the concavity of $d_{xy}$, $s_{xy}$ and $f_x$ and the convexity of $c_{xy}$, as well as the feasibility of the problem resulting from \eqref{fea:eqn1} and \eqref{fea:eqn2}.

\subsection{Distributed Algorithm}
In this subsection, we aim to develop a distributed algorithm to solve the formulated problem. Our first step is to rewrite the optimization problem in the ADMM form by introducing ancillary variables $\pi_{xy}^{t,d}$ and $\pi_{xy}^{t,s}$. The additional superscripts $d$ and $s$ indicate that the corresponding parameters belong to the destination/target node or the source node, respectively. We then set $\pi_{xy}^t = \pi_{xy}^{t,d}$ and $\pi_{xy}^t = \pi_{xy}^{t,s}$, indicating that the solutions proposed by the targets and sources are consistent with the ones proposed by the central planner. This reformulation facilitates the design of a distributed algorithm which allows us to iterate through the process in obtaining the fair and efficient transport plan. To this end, the reformulated optimal transport problem under fairness consideration is presented as follows:
\begin{align}
\min_{\Pi_d \in \mathcal{F}_d, \Pi_s \in \mathcal{F}_s,\Pi} & -\sum_{t\in\mathcal{T}}\sum_{x\in\mathcal{X}} \sum_{y\in\mathcal{Y}_x} d_{xy}(\pi_{xy}^{t,d}) -\sum_{t\in\mathcal{T}} \sum_{y\in\mathcal{Y}} \sum_{x\in\mathcal{X}_y} (s_{xy}(\pi_{xy}^{t,s}) \notag\\
&- c_{xy}(\pi_{xy}^{t,s}))  -\sum_{x\in\mathcal{X}}\omega_x f_{x}(\sum_{t\in\mathcal{T}}\sum_{y\in\mathcal{Y}_x}\pi_{xy}^{t,d})\notag\\
\mathrm{s.t.}\quad & \pi_{xy}^{t,d} = \pi_{xy}^t,\ \forall t\in\mathcal{T},\ \forall \{x,y\}\in\mathcal{E},\label{OT2:eqn}\\
& \pi_{xy}^t=\pi_{xy}^{t,s},\ \forall t\in\mathcal{T},\ \forall \{x,y\}\in\mathcal{E},\notag
\end{align}
where $\Pi_d:=\{\pi_{xy}^{t,d}\}_{x\in\mathcal{X}_y,y\in\mathcal{Y},t\in\mathcal{T}}$, $\Pi_s:=\{\pi_{xy}^{t,s}\}_{x\in\mathcal{X},y\in\mathcal{Y}_x,t\in\mathcal{T}}$, $\mathcal{F}_d := \{ \Pi_d | \pi_{xy}^{t,d} \geq 0, \underline{p}_x \leq \sum_{t\in\mathcal{T}} \sum_{y \in \mathcal{Y}_x} \pi_{xy}^{t,d} \leq \bar{p}_x,\{x,y\} \in \mathcal{E},t\in\mathcal{T}\}$, and $\mathcal{F}_s := \{ \Pi_s | \pi_{xy}^{t,s} \geq 0, \underline{q}_y \leq \sum_{t\in\mathcal{T}}\sum_{x \in \mathcal{X}_y} \pi_{xy}^{t,s} \leq \bar{q}_y,\{x,y\} \in \mathcal{E},t\in\mathcal{T} \}$.

Note that we transform the original maximization of the social utility problem to an equivalent program of minimizing the aggregated cost. Furthermore, due to the constraints, the optimal solutions of $\Pi_t$, $\Pi_s$, and $\Pi$ to \eqref{OT2:eqn} are the same.
Our next focus is to develop a distributed algorithm to solve the problem \eqref{OT2:eqn}. We let $\alpha_{xy}^{t,s}$ and $\alpha_{xy}^{t,d}$ be the Lagrangian multipliers associated with the constraint $\pi_{xy}^{t,s} = \pi_{xy}^t$ and $\pi_{xy}^t=\pi_{xy}^{t,d}$, respectively. The Lagrangian then facilitates the application of ADMM in the distributed algorithm design. Specifically, the Lagrangian associated with the optimization problem \eqref{OT2:eqn} can then be written as follows:
\begin{align}
    L &\left(\Pi_{d}, \Pi_{s}, \Pi, \alpha_{xy}^{t,s}, \alpha_{xy}^{t,d} \right) = \notag\\
    &- \sum_{t\in\mathcal{T}}\sum_{x\in\mathcal{X}} \sum_{y\in\mathcal{Y}_x} d_{xy}(\pi_{xy}^{t,d}) - \sum_{t\in\mathcal{T}}\sum_{y\in\mathcal{Y}} \sum_{x\in\mathcal{X}_y} \left(s_{xy}(\pi_{xy}^{t,s}) - c_{xy}(\pi_{xy}^{t,s})\right) \notag\\
    & -\sum_{x\in\mathcal{X}}\omega_x f_{x}(\sum_{t\in\mathcal{T}}\sum_{y\in\mathcal{Y}_x}\pi_{xy}^{t,d})+ \sum_{t\in\mathcal{T}}\sum_{x\in\mathcal{X}} \sum_{y\in\mathcal{Y}_x} \alpha_{xy}^{t,d} (\pi_{xy}^{t,d} - \pi_{xy}^t) \notag\\ &+\sum_{t\in\mathcal{T}}\sum_{y\in\mathcal{Y}} \sum_{x\in\mathcal{X}_y} \alpha_{xy}^{t,s} (\pi_{xy}^t - \pi_{xy}^{t,s}) + \frac{\eta}{2} \sum_{t\in\mathcal{T}}\sum_{x\in\mathcal{X}} \sum_{y\in\mathcal{Y}_x} (\pi_{xy}^{t,d} - \pi_{xy}^t)^2 \notag\\
    &+ \frac{\eta}{2} \sum_{t\in\mathcal{T}}\sum_{y\in\mathcal{Y}} \sum_{x\in\mathcal{X}_y} (\pi_{xy}^t - \pi_{xy}^{t,s})^2,\label{Lag:eqn}
\end{align}
where $\eta > 0$ is a positive scalar constant controlling the convergence rate in the algorithm designed below.
In \eqref{Lag:eqn}, the last two terms $\frac{\eta}{2}\sum_{t\in\mathcal{T}} \sum_{x\in\mathcal{X}} \sum_{y\in\mathcal{Y}_x} (\pi_{xy}^{t,d} - \pi_{xy}^t)^2$ and $\frac{\eta}{2} \sum_{t\in\mathcal{T}}\sum_{y\in\mathcal{Y}} \sum_{x\in\mathcal{X}_y} (\pi_{xy}^t - \pi_{xy}^{t,s})^2$, acting as penalization, are quadratic. Hence, the Lagrangian function $L$ is strictly convex, ensuring the existence of a unique optimal solution. 
We can apply ADMM to the minimization problem in \eqref{OT2:eqn}. The designed distributed algorithm is presented in the following proposition.

\begin{proposition}
 The iterative steps of ADMM to \eqref{OT2:eqn} are summarized as follows:
\begin{equation}\label{ADMM1_eqn1}
\begin{split}
    \Pi_{x}^d(k+1) \in \arg \min_{\Pi_{x}^d\in\mathcal{F}_{x}^d} - \sum_{t\in\mathcal{T}}\sum_{y\in\mathcal{Y}_x} d_{xy}(\pi_{xy}^{t,d}) - \omega_x f_{x}(\sum_{t\in\mathcal{T}}\sum_{y\in\mathcal{Y}_x}\pi_{xy}^{t,d}) \\ + \sum_{t\in\mathcal{T}}\sum_{y\in\mathcal{Y}_x} \alpha_{xy}^{t,d}(k) \pi_{xy}^{t,d} + \frac{\eta}{2}\sum_{t\in\mathcal{T}} \sum_{y\in\mathcal{Y}_x} (\pi_{xy}^{t,d} - \pi_{xy}^t(k))^2,
\end{split}
\end{equation}
\begin{equation}\label{ADMM1_eqn2}
\begin{aligned}
        \Pi_{y}^s(k+1) \in \arg \min_{\Pi_{y}^s\in\mathcal{F}_{y}^s} - \sum_{t\in\mathcal{T}}\sum_{x\in\mathcal{X}_y} \left(s_{xy}(\pi_{xy}^{t,s}) - c_{xy}(\pi_{xy}^{t,s})\right) \\ -\sum_{t\in\mathcal{T}}\sum_{x\in\mathcal{X}_y} \alpha_{xy}^{t,s}(k)\pi_{xy}^{t,s} + \frac{\eta}{2} \sum_{t\in\mathcal{T}}\sum_{x\in\mathcal{X}_y} (\pi_{xy}^t(k) - \pi_{xy}^{t,s})^2,
\end{aligned}
\end{equation}
\begin{equation}\label{ADMM1_eqn3}
\begin{split}
    \Pi_{xy}(&k+1)= \arg \min_{\Pi_{xy}} - \sum_{t\in\mathcal{T}} \alpha_{xy}^{t,d}(k)\pi_{xy}^t + \sum_{t\in\mathcal{T}}\alpha_{xy}^{t,s}(k)\pi_{xy}^t \\
    &+\frac{\eta}{2}\sum_{t\in\mathcal{T}}(\pi_{xy}^{t,d}(k+1) - \pi_{xy}^t)^2 + \frac{\eta}{2}\sum_{t\in\mathcal{T}}(\pi_{xy}^t - \pi_{xy}^{t,s}(k+1))^2,
\end{split}
\end{equation}
\begin{equation}\label{ADMM1_eqn4}
\begin{split}
    \alpha_{xy}^{t,d}(k+1) = \alpha_{xy}^{t,d}(k) + \eta(\pi_{xy}^{t,d}(k+1)-\pi_{xy}^t(k+1))^2,
\end{split}
\end{equation}
\begin{equation}\label{ADMM1_eqn5}
\begin{split}
    \alpha_{xy}^{t,s}(k+1) = \alpha_{xy}^{t,s}(k) + \eta(\pi_{xy}^t(k+1)-\pi_{xy}^{t,s}(k+1))^2,
\end{split}
\end{equation}
where $\Pi_{\tilde{x}}^d:=\{\pi_{xy}^{t,d}\}_{y\in\mathcal{Y}_x,x=\tilde{x},t\in\mathcal{T}}$ represents the solution at target node $\tilde{x}\in\mathcal{X}$, $\Pi_{\tilde{y}}^s:=\{\pi_{xy}^{t,s}\}_{x\in\mathcal{X}_y,y=\tilde{y},t\in\mathcal{T}}$ represents the proposed solution at source node $\tilde{y}\in\mathcal{Y}$, and $\Pi_{\tilde{x}\tilde{y}}:=\{\pi_{xy}^{t}\}_{y=\tilde{y},x=\tilde{x},t\in\mathcal{T}}$ includes the solution between $\tilde{x}$ and $\tilde{y}$. In addition, $\mathcal{F}_{x}^d := \{ \Pi_{x}^d | \pi_{xy}^{t,d} \geq 0, y\in\mathcal{Y}_x, t\in\mathcal{T}, \underline{p}_x \leq \sum_{t\in\mathcal{T}}\sum_{y \in \mathcal{Y}_x} \pi_{xy}^{t,d} \leq \bar{p}_x\}$, and $\mathcal{F}_{y}^s := \{ \Pi_{y}^s | \pi_{xy}^{t,s} \geq 0, x\in\mathcal{X}_y, t\in\mathcal{T}, \underline{q}_y \leq \sum_{t\in\mathcal{T}}\sum_{x \in \mathcal{X}_y} \pi_{xy}^{t,s} \leq \bar{q}_y\}$.
\end{proposition}

\begin{proof}
Let $\Vec{x} = [\Vec{\Pi}^{d\mathsf{T}}, \Vec{\Pi}^\mathsf{T}]^\mathsf{T}$, $\Vec{y} = [\Vec{\Pi}^\mathsf{T}, \Vec{\Pi}^{s\mathsf{T}}]^\mathsf{T}$, and $\alpha = [\{\Vec{\alpha_{xy}^{t,s}}\}^\mathsf{T}, \{\Vec{\alpha_{xy}^{t,d}}\}^\mathsf{T}]^\mathsf{T}$, where $\mathsf{T}$ and $\Vec{}$ denotes the transpose and vectorization operator. Note that these three vectors are all $2T|\mathcal{E}| \times 1$. Now we can write the constraints in \eqref{OT2:eqn} in a matrix form such that $\mathbf{A}\Vec{x} = \vec{y}$, where $\mathbf{A} = [\textbf{I},\textbf{0};\textbf{I},\textbf{0}]$ with $\textbf{I}$ and $\textbf{0}$ denoting the $T|\mathcal{E}|$-dimensional identity and zero matrices, respectively. Next, we note that $\Vec{x} \in \mathcal{F}_{\Vec{x}}^d$ and $\Vec{y} \in \mathcal{F}_{\Vec{y}}^s$, where
$
      \mathcal{F}_{\Vec{x}}^d = \{ \Vec{x} | \pi_{xy}^{t,d} \geq 0, \underline{p}_x \leq \sum_{t\in\mathcal{T}}\sum_{y \in \mathcal{Y}_x} \pi_{xy}^{t,d} \leq \bar{p}_x, \{x,y\} \in \mathcal{E} \},\ 
       \mathcal{F}_{\Vec{y}}^s := \{ \Vec{y} | \pi_{xy}^{t,s} \geq 0, \underline{q}_y \leq \sum_{t\in\mathcal{T}}\sum_{x \in \mathcal{X}_y} \pi_{xy}^{t,s} \leq \bar{q}_y, \{x,y\} \in \mathcal{E}  \}.
$
Then, we can solve \eqref{OT2:eqn} using the iterations: 1)
$
    \Vec{x}(k+1) \in \arg \min_{\Vec{x} \in \mathcal{F}_{\Vec{x}}^d} L(\Vec{x},\Vec{y}(k),\alpha(k));
$
2)
$
    \Vec{y}(k+1) \in \arg \min_{\Vec{y} \in \mathcal{F}_{\Vec{y}}^s} L(\Vec{x}(k+1),\Vec{y},\alpha(k));
$
3)
$
    \alpha(k+1) = \alpha(k) + \eta(A\Vec{x}(k+1) - \Vec{y}(k+1)),
$
whose convergence is proved \cite{boyd2011distributed}.
Because we have no coupling among $\Pi_{x}^d, \Pi_{y}^s, \Pi_{xy}, \alpha_{xy}^{t,d},$ and $\alpha_{xy}^{t,s}$, the above iterations can be equivalently decomposed to  \eqref{ADMM1_eqn1}-\eqref{ADMM1_eqn5}. 
\end{proof}
We can further simplify equations \eqref{ADMM1_eqn1}-\eqref{ADMM1_eqn5} down to four equations, and the results are summarized below.

\begin{proposition}\label{prop:2}
The iterations \eqref{ADMM1_eqn1}-\eqref{ADMM1_eqn5} can be simplified as follows:
\begin{equation}\label{ADMM2_eqn1}
\begin{split}
    \Pi_{x}^d(k+1) \in \arg \min_{\Pi_{x}^d\in\mathcal{F}_{x}^d} - \sum_{t\in\mathcal{T}}\sum_{y\in\mathcal{Y}_x} d_{xy}(\pi_{xy}^{t,d}) - \omega_x f_{x}(\sum_{t\in\mathcal{T}}\sum_{y\in\mathcal{Y}_x}\pi_{xy}^{t,d}) \\ + \sum_{t\in\mathcal{T}}\sum_{y\in\mathcal{Y}_x} \alpha_{xy}^{t}(k) \pi_{xy}^{t,d} + \frac{\eta}{2}\sum_{t\in\mathcal{T}} \sum_{y\in\mathcal{Y}_x} (\pi_{xy}^{t,d} - \pi_{xy}^t(k))^2,
\end{split}
\end{equation}
\begin{equation}\label{ADMM2_eqn2}
\begin{split}
        \Pi_{y}^s(k+1) \in \arg \min_{\Pi_{y}^s\in\mathcal{F}_{y}^s} - \sum_{t\in\mathcal{T}}\sum_{x\in\mathcal{X}_y} \left(s_{xy}(\pi_{xy}^{t,s}) - c_{xy}(\pi_{xy}^{t,s})\right) \\ +\sum_{t\in\mathcal{T}}\sum_{x\in\mathcal{X}_y} \alpha_{xy}^{t}(k)\pi_{xy}^{t,s} + \frac{\eta}{2} \sum_{t\in\mathcal{T}}\sum_{x\in\mathcal{X}_y} (\pi_{xy}^t(k) - \pi_{xy}^{t,s})^2,
\end{split}
\end{equation}
\begin{equation}\label{ADMM2_eqn3}
\begin{split}
    \Pi_{xy}(k+1) = \frac{1}{2} \left(\Pi_{xy}^{d}(k+1) + \Pi_{xy}^{s}(k+1)\right),
\end{split}
\end{equation}
\begin{equation}\label{ADMM2_eqn4}
\begin{split}
    \alpha_{xy}^t(k+1) = \alpha_{xy}^t(k) + \frac{\eta}{2}\left(\pi_{xy}^{t,d}(k+1) - \pi_{xy}^{t,s}(k+1)\right),
\end{split}
\end{equation}
where $\Pi_{xy}^{d}$ and $\Pi_{xy}^{s}$ in \eqref{ADMM2_eqn3} are obtained from \eqref{ADMM2_eqn1} for fixed $y$ and \eqref{ADMM2_eqn2} for fixed $x$, respectively.
\end{proposition}

\begin{proof}
As \eqref{ADMM1_eqn3} is strictly concave, we can solve it by first-order condition:
$
    \pi_{xy}^t(k+1) = \frac{1}{2\eta}(\alpha_{xy}^{t,d}(k) - \alpha_{xy}^{t,s}(k)) + \frac{1}{2}(\pi_{xy}^{t,d}(k+1) + \pi_{xy}^{t,s}(k+1)).
$
By substituting the above equation into  \eqref{ADMM1_eqn4} and \eqref{ADMM1_eqn5} we get:
$
    \alpha_{xy}^{t,d}(k+1) = \frac{1}{2}(\alpha_{xy}^{t,d}(k) + \alpha_{xy}^{t,s}(k)) + \frac{\eta}{2}(\pi_{xy}^{t,d}(k+1) - \pi_{xy}^{t,s}(k+1)),
$
$
    \alpha_{xy}^{t,s}(k+1) = \frac{1}{2}(\alpha_{xy}^{t,d}(k) + \alpha_{xy}^{t,s}(k)) + \frac{\eta}{2}(\pi_{xy}^{t,d}(k+1) - \pi_{xy}^{t,s}(k+1)).
$
We can see that $\alpha_{xy}^{t,d} = \alpha_{xy}^{t,s}$ during each update. Hence, $\pi_{xy}^t(k+1)$ can be further simplified as $\pi_{xy}^t(k+1) = \frac{1}{2} (\pi_{xy}^{t,d}(k+1) + \pi_{xy}^{t,s}(k+1))$ shown in \eqref{ADMM2_eqn3}. In addition, we can achieve \eqref{ADMM1_eqn4} and \eqref{ADMM1_eqn5} from  $\alpha_{xy}^{t,d} = \alpha_{xy}^{t,s} = \alpha_{xy}^t$ represented in \eqref{ADMM2_eqn4}.
\end{proof}

We can iterate through equations \eqref{ADMM2_eqn1}-\eqref{ADMM2_eqn4} to obtain a fair and efficient resource transport strategy until getting a convergence. Note that the fairness is explicitly considered during the solution updates, which can be seen from the ADMM iteration step \eqref{ADMM2_eqn1}. For convenience, we summarize the iterations in Proposition \ref{prop:2}  in Algorithm \ref{Alg:1}.

\begin{algorithm}[!t]
\caption{Distributed Algorithm}\label{Alg:1}
\begin{algorithmic}[1]
\While {$\Pi_{x}^d$ and $\Pi_{y}^s$ not converging}
\State Compute $\Pi_{x}^d(k+1)$  using \eqref{ADMM2_eqn1},  $\forall x\in\mathcal{X}_y$
\State Compute $\Pi_{y}^s(k+1)$  using \eqref{ADMM2_eqn2}, $\forall y\in\mathcal{Y}_x$
\State Compute $\Pi_{xy}(k+1)$  using \eqref{ADMM2_eqn3}, $\forall \{x,y\}\in \mathcal{E}$
\State Compute $\alpha_{xy}^t(k+1)$  using \eqref{ADMM2_eqn4}, $\forall t\in\mathcal{T}$, $\forall \{x,y\}\in \mathcal{E}$
\EndWhile
\State \textbf{return} $\pi_{xy}^t(k+1)$, $\forall t\in\mathcal{T}$, $\forall \{x,y\}\in \mathcal{E}$
\end{algorithmic}
\end{algorithm}

\section{Discussions on the Distributed Algorithm}\label{sec:discussion}
In this section, we discuss several crucial aspects of the proposed distributed algorithm for fair and efficient resource allocation mechanisms.

 \subsection{Fairness and Efficiency Trade-off}
The fairness of the transport scheme is ensured during the updates of solutions. As shown in \eqref{ADMM2_eqn1}, the level of fairness is regulated by the parameter $\omega_x$, $x\in\mathcal{X}$. Specifically, $\omega_x$ trades off between the efficiency and fairness of the transport strategy. With a larger $\omega_x$, the fairness term has a more significant impact on the solution, yielding a fairer resource allocation plan but in turn the plan is less efficient. For every target $x$, it maximizes $f_{x}(\sum_{t\in\mathcal{T}}\sum_{y\in\mathcal{Y}_x}\pi_{xy}^{t,d})$ at each step. The concavity of $f_x$ guarantees that it is impossible for a single target in the network to receive all the resources. Together with the penalization terms $\sum_{t\in\mathcal{T}}\sum_{y\in\mathcal{Y}_x} \alpha_{xy}^{t}(k) \pi_{xy}^{t,d} + \frac{\eta}{2}\sum_{t\in\mathcal{T}} \sum_{y\in\mathcal{Y}_x} (\pi_{xy}^{t,d} - \pi_{xy}^t(k))^2$, it also ensures that the request for resources from each target $x$ will not be arbitrarily large.


\subsection{Implementation of Fairness }
In the reformulated problem \eqref{OT2:eqn}, we associated the fairness function, $f_x$, $\forall x\in\mathcal{X}$, with the corresponding target node. This leads to natural interpretations that when proposing the transport strategy, each target needs to be aware of the fairness of the resource allocation over networks. In a resource distribution market, the supplier (source node) may not care where its resources are allocated in the end. However, a target cares whether it gets more or less resources than another target. For example, if a large company is distributing resources to customers, the company does not care where their product goes, while consumers care if only a few consumers are able to obtain the product. This observation is consistent with the iteration steps \eqref{ADMM2_eqn1} and \eqref{ADMM2_eqn2}, where each target $x$ aims to maximize the fairness term $\omega_x f_{x}(\sum_{t\in\mathcal{T}}\sum_{y\in\mathcal{Y}_x}\pi_{xy}^{t,d}) $, while each source $y$ merely maximizes its own utility. 
Note that during the problem reformulation, the fairness term could also be applied to the source. 
We can design distributed algorithm to solve this reformulated problem using similar techniques as in Section \ref{sec:algorithm}. 

\subsection{Continuous and Distributed Resource Allocation}\label{subsec:online}
In Algorithm \ref{Alg:1}, all the participants (sources and targets) updates their decisions on the transferred/requested resources iteratively in a distributed fashion. In a resource distribution market, the number of participants and their preferences can vary over time. For example, some suppliers will leave the market when they finish the allocation of their resources. Similarly, new target nodes may join the market when they need to purchase resources. Hence, it is necessary to devise a continuous resource allocation mechanism that is adaptive to the changes in the market.
We can extend the Algorithm \ref{Alg:1} in this regard and implement it in an online form. Specifically, when there are changes in the market, we can continue to solve the optimal transport problem by using Algorithm \ref{Alg:1} with necessary updates resulting from the market changes. In this way, we do not need to recompute the fair and efficient transport strategy for the new scenario from scratch. The algorithm will take into account these changes inherently and continuously update the solution in a distributed way. 
We will illustrate the continuous resource allocation with a case study in Section \ref{sec:case}.

\section{Case Studies}\label{sec:case}
In this section, we corroborate our algorithm for distributed optimal transport with fairness consideration. We consider a scenario with five target nodes and two source nodes and a transport network structure connecting all source nodes to both target nodes.  The upper bounds are $\bar{q}_1 = 2$, $\bar{q}_2 = 3$, $\bar{q}_3 = 4$, $\bar{q}_4 = 3$, $\bar{q}_5 = 2$, $\bar{p}_1 = 4$, and $\bar{p}_2 = 4$.
The lower bounds, $\underline{q}_y$ and $\underline{p}_x$ are 0 for all nodes. For illustration simplicity, we consider the resource allocation over a single period, i.e., $T=1$. Thus, we omit the time index $t$ in the notations in case studies. Furthermore, we adopt the following linear utility and cost functions: $d_{xy}(\pi_{xy}) = \delta_{xy}\pi_{xy}$, $s_{xy}(\pi_{xy}) = \sigma_{xy}\pi_{xy}$ and $c_{xy}(\pi_{xy}) = \zeta_{xy}\pi_{xy}$, $\forall \{x,y\} \in \mathcal{E}$. The corresponding parameters are selected as:
$
    [\zeta_{xy}]_{x\in\mathcal{X},y\in\mathcal{Y}} = \begin{bmatrix} 1 & 2 & 1 & 2 & 1 \\ 2 & 1 & 3 & 1 & 2\end{bmatrix}.
$
We consider proportional fairness in the resource allocation, i.e., the fairness function admits the form shown in \eqref{fairness:eqn}. 


\subsection{Fair and Distributed Resource Allocation}
We first show the effectiveness of the designed Algorithm \ref{Alg:1}.
Specifically, we compare the optimal transport strategies with and without fairness considerations using Algorithm \ref{Alg:1}. For the algorithm with fairness, we set the weighting factor $\omega_x = 3$. We focus on comparing their induced social utility. The social utility is the aggregate of the payoffs of the sources and targets and the benefits of fairness in resource allocation. The results are shown in Fig. \ref{fig:f1}. Fig. \ref{fig:f1_1} indicates that the distributed algorithm (both with and without fairness consideration) converges to the corresponding centralized optimal solution $\pi_{xy}^o$ (i.e., problem \eqref{OT2:eqn} is solved directly). We also observe in Fig. \ref{fig:f1_1} that the algorithm with fairness converges to a higher social utility. The increase in the social utility is due to the addition of fairness when designing the resource transport scheme. We also note that the fairness has little effect on the convergence of the algorithm. Fig. \ref{fig:f1_2} shows the residual of transport strategy. The residual measures the difference between the strategy at the current update and the centralized optimal solution. We can observe that the residual goes to 0 around $k=50$, which demonstrates the effectiveness of the designed distributed algorithm.

After verifying that the algorithm works and increases social utility, we further verify that it is efficient in computing the transport strategy. In a case with 20 source nodes and 20 target nodes, the designed distributed algorithm takes just over two minutes which is considerably good given that there are $20\times 20$ connections with the transport scheme where all source nodes are connected to all target nodes.


\begin{figure}[!t] 
\centering
\subfigure[Social utility ]{\includegraphics[width=0.48\columnwidth]{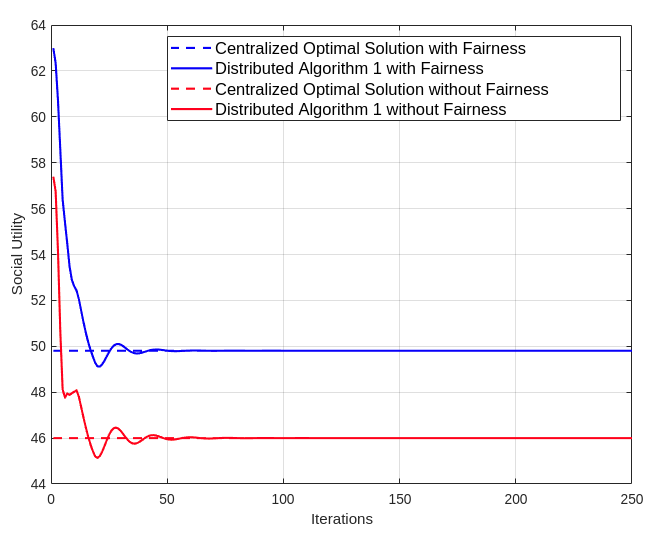}\label{fig:f1_1}}
\subfigure[$\sqrt{\sum_{x\in\mathcal{X}}\sum_{y\in\mathcal{Y}_x}(\pi_{xy}(k) - \pi_{xy}^o)^2}$]{\includegraphics[width=0.5\columnwidth]{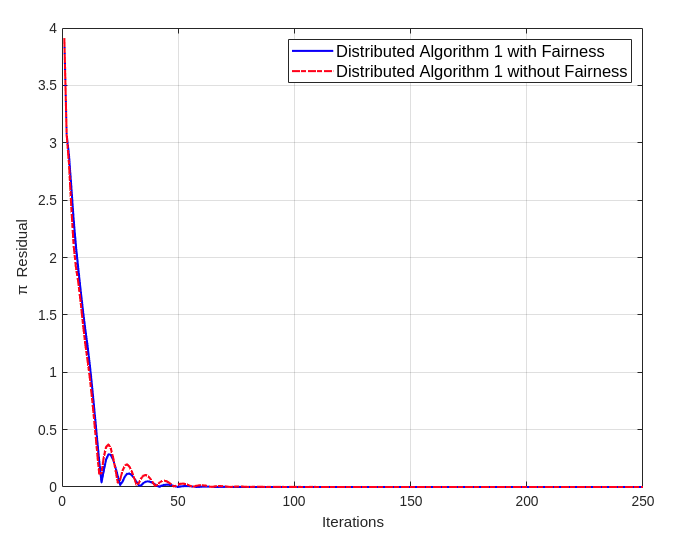}\label{fig:f1_2}}
\caption{Impact of fairness consideration on the transport strategy design using Algorithm \ref{Alg:1}. (a) and (b) depict the trajectories of social utility and residual of transport strategy, respectively.}
\label{fig:f1}
\vspace{-4mm}
\end{figure}

\subsection{Online Distributed Resource Allocation}
Next, we investigate a case study using the discussed continuous, or online, distributed algorithm in Section \ref{subsec:online}. We adopt the same utility, cost and fairness functions as in the previous scenario. In this case, the resource allocation network changes over time as shown in Fig. \ref{fig:f2}. Specifically, there are three sources and two targets at $k=0$, and not all of which are connected, i.e., the bipartite graph is incomplete. At step $k=250$, one target node joins the network, and hence the network has three source nodes and three target nodes. At step $k=500$, one source node leaves the network, and hence two source nodes needs to satisfy the requests from three target nodes. When the resource allocation network is changed, the upper bounds on the amount of transferable resources at sources and the amount of sources requested at the targets are also updated, with each parameter shown in Fig. \ref{fig:f2}. The weighting constant on the fairness is chosen as $\omega_x = 3$, $\forall x\in\mathcal{X}$, throughout the case study. Other parameters are summarized in Table \ref{tab:parameters}. The online algorithm addresses the problem continuously without resetting the algorithm. The results are shown in Fig. \ref{fig:f2}. When the resource transport network changes (at $k=250,\ 500$), the online algorithm will respond to these changes quickly by proposing new allocation schemes. The solutions obtained from the online distributed algorithm are consistent with the centralized optimal solutions. Thus, this online distributed algorithm is applicable to the resource distribution market with frequent changes.

\begin{figure}[!t] 
\centering
\subfigure[$k = 0$]{\includegraphics[width=0.3\columnwidth]{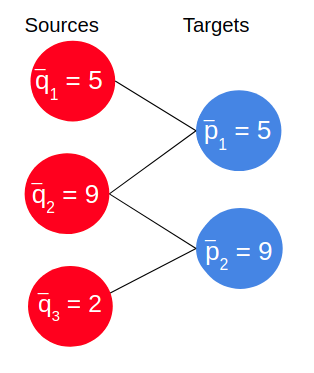}\label{fig:net_on1}}
\subfigure[$k = 250$]{\includegraphics[width=0.3\columnwidth]{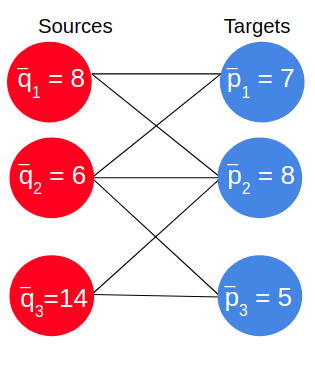}\label{fig:net_on2}}
\subfigure[$k = 500$]{\includegraphics[width=0.3\columnwidth]{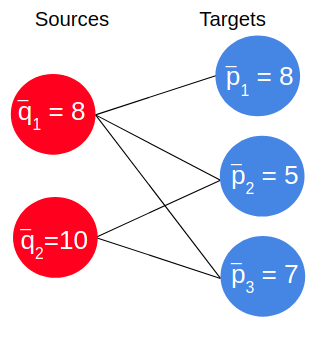}\label{fig:net_on3}}
\caption{Network structures for the continuous/online resource allocation. }
\label{fig:online_net}
\vspace{-3mm}
\end{figure}

\begin{table}
\centering
\caption{Parameters in the Online Distributed Resource Allocation}
\label{tab:parameters}
\begin{tabular}{ |c|c|c|c|c| } 
\hline
\multirow{5}{4em}{$k=0$} & $\delta_{11}=2$ & $\delta_{21}=3$ & $\delta_{22}=4$ & $\delta_{32}=4$ \\ 
& $\sigma_{11}=7$ & $\sigma_{21}=8$ & $\sigma_{22}=6$ & $\sigma_{32}=5$ \\
& $\zeta_{11}=2$ & $\zeta_{21}=4$ & $\zeta_{22}=2$ & $\zeta_{32}=2$ \\
& $\delta_{13}=0$ & $\sigma_{13} = 0$ & $\zeta_{13} = 0 $ & $\delta_{31} = 0$ \\
& $\sigma_{31} = 0$ &$\zeta_{31} = 0$ & &\\
\hline
\multirow{7}{4em}{$k=250$} & $\delta_{11}=4$ & $\delta_{12}=2$ & $\delta_{13}=0$ & $\delta_{21}=3$ \\ 
& $\delta_{22}=5$ & $\delta_{23}=4$ & $\delta_{31}=0$ & $\delta_{32}=6$ \\
& $\delta_{33}=2$ & $\sigma_{11}=8$ & $\sigma_{12}=14$ & $\sigma_{13}=0$ \\
& $\sigma_{21}=7$ & $\sigma_{22}=10$ & $\sigma_{23}=9$ & $\sigma_{31}=0$ \\
& $\sigma_{32}=12$ & $\sigma_{33}=4$ & $\zeta_{11}=2$ & $\zeta_{12}=10$ \\
& $\zeta_{13}=0$ & $\zeta_{21}=4$ & $\zeta_{22}=5$ & $\zeta_{23}=5$ \\
& $\zeta_{31}=0$ & $\zeta_{32}=6$ & $\zeta_{33}=2$ & \\
\hline
\multirow{5}{4em}{$k=500$} & $\delta_{11}=3$ & $\delta_{12}=2$ & $\delta_{13}=5$ & $\delta_{22}=3$ \\ 
& $\delta_{23}=3$ & $\sigma_{11}=5$ & $\sigma_{12}=7$ & $\sigma_{13}=5$ \\
& $\sigma_{22}=7$ & $\sigma_{23}=4$ & $\zeta_{11}=2$ & $\zeta_{12}=2$ \\
& $\zeta_{13}=1$ & $\zeta_{22}=1$ & $\zeta_{23}=2$ & $\delta_{21} = 0$\\
& $\sigma_{21} = 0$ & $\zeta_{21} = 0$ & & \\
\hline
\end{tabular}
\vspace{-5mm}
\end{table}

\begin{figure}[!t] 
\centering
\subfigure[Social utility  ]{\includegraphics[width=0.48\columnwidth]{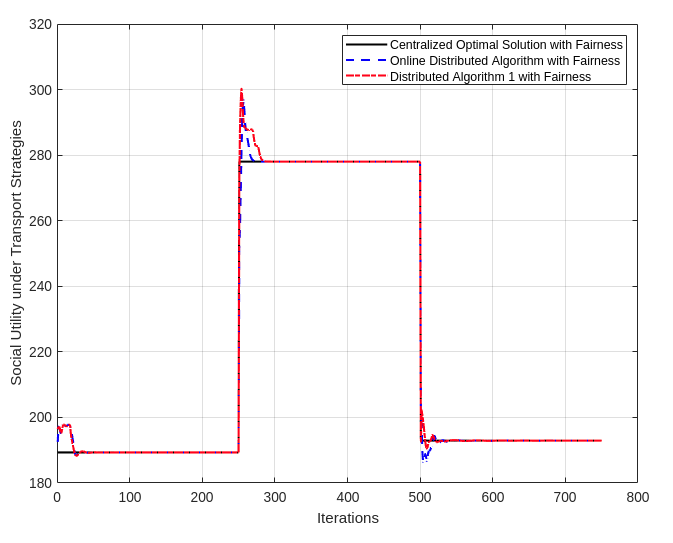}\label{fig:f2_1}}
\subfigure[$\sqrt{\sum_{x\in\mathcal{X}}\sum_{y\in\mathcal{Y}_x}(\pi_{xy}(k) - \pi_{xy}^o)^2}$]{\includegraphics[width=0.49\columnwidth]{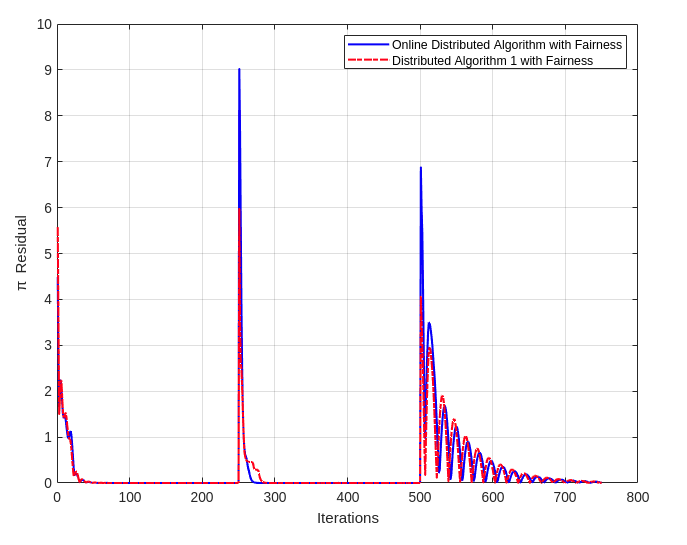}\label{fig:f2_2}}
\caption{Adaptive fair and efficient transport strategies design using the online algorithm. The transport network structure and participants preferences change over time at $k=250,\ 500$. (a) and (b) depict the trajectories of social utility and residual of transport strategy, respectively.}
\label{fig:f2}
\vspace{-5mm}
\end{figure}



\section{Conclusion}\label{sec:conclusion}
In this paper, we have investigated fair and efficient dynamic transport of limited amount of resources in a network of participants with various preferences. The designed distributed algorithm can successfully yield the identical transport plan designed under the centralized manner, making our algorithm applicable to large-scale networks. The fairness is explicitly promoted in the algorithm, through bargaining and negotiations between each pair of resource supplier (source) and resource receiver (target). Throughout the negotiation, the sources maximize their revenue but need to consider the fairness. Similarly, the targets optimize the fairness but should take into account the efficiency of resource allocation. The algorithm terminates when the two parties reach a consensus.  Further work includes the investigation of fair resource transport under incomplete information between two parties.



\bibliographystyle{IEEEtran}
\bibliography{references.bib}

\end{document}